\documentclass[a4paper,12pt]{article}

\usepackage{palatino}
\usepackage{mathpazo}

\usepackage{mathrsfs}
\usepackage{latexsym,amssymb,amsmath,amsthm,url}

\usepackage{float}
\usepackage{enumerate}

\theoremstyle{plain} 

\newtheorem{lemma}{Lemma}[section]
\newtheorem{thm}[lemma]{Theorem}

\newtheorem{cor}[lemma]{Corollary}

\theoremstyle{definition} 

\newtheorem{definition}[lemma]{Definition}

\theoremstyle{remark}
\newtheorem{remark}[lemma]{Remark}
 
\newcommand{\N}{\mathbb{N}}

\title{Concrete algorithms for word problem and subsemigroup problem for semigroups which are disjoint unions of finitely many copies of the free monogenic semigroup}
\author{\textbf{Nabilah Abughazalah\thanks{The author is financially supported by Princess Nourah bint Abdulrahman University in Riyadh and  Saudi Aramco Ibn Khaldun Fellowship for Saudi Women, in partnership with the Center for Clean Water and Clean Energy at MIT.
}}\\
Department of Mathematical Sciences\\ Princess Nourah bint Abdulrahman University\\ Riyadh, Saudi Arabia\\ E-mail: nhabughazala@pnu.edu.sa }

\begin{document}

\maketitle

\begin{abstract}
Every semigroup which is a finite disjoint union of copies of the free monogenic
semigroup (natural numbers under addition) has  soluble word problem and  soluble membership problem. Efficient algorithms are given for both problems.
\bigskip
\noindent
\end{abstract}
\bigskip

\section{Introduction}
It is well known that some semigroups may be decomposed into a disjoint union of subsemigroups which is unlike the structures of classical algebra such as groups and rings. For instance, the Rees Theorem states that every completely simple semigroup is a Rees matrix semigroup over a group $G$, and is thus a disjoint union of copies of $G$, see \cite[Theorem 4.2.1]{howie95}; every Clifford semigroup is a strong semilattice of groups and as such it is a disjoint union of its maximal subgroups, see \cite[Theorem IV.2.2]{howie95}; every commutative semigroup is a semilattice of archimedean semigroups, see \cite[Theorem 3.3.1]{grillet95}.

\quad  If $S$ is a semigroup which can be decomposed into a disjoint union of subsemigroups, then it is natural to ask how the properties of $S$ depend on these  subsemigroups. For
example, if the subsemigroups are finitely generated, then so is $S$. Ara\'{u}jo et al.\cite{araujo01} consider the finite presentability of semigroups which are disjoint unions of finitely presented subsemigroups; Golubov \cite {golubov75} showed that a semigroup which is a disjoint union of residually finite subsemigroups is residually finite.

\quad In the context where $S$ is a semigroup which is a disjoint union of finitely many copies of the free monogenic semigroup, the authors in \cite {abughazalah13} proved
that $S$ is finitely presented and residually finite; in \cite {abughazalah14} the authors proved that, up to isomorphism and anti-isomorphism, there are only two types of semigroups which are unions of two copies of the free monogenic
semigroup. Similarly, they showed that there are only nine types of semigroups which are unions of three copies of the free monogenic semigroup and provided finite presentations for semigroups of each of these types. 

\quad In this paper we continue investigating finiteness conditions for a semigroup which is a disjoint union of finitely many copies of the free monogenic semigroup, the decidability of the word problem and membership problem in particular. 

\quad The paper is organized as follows. In section 2 we recall some lemmas from \cite {abughazalah13} and explain the obtained results with clarify the strong regularities which are all described in terms of arithmetic progressions. In Section 3 we prove that $S$ has a soluble word problem and  soluble membership problem.

\section{Properties of the semigroup which is a disjoint union of finitely many copies of the free monogenic semigroup}
Let $S$ be a semigroup which is a disjoint union of finitely many copies of the free monogenic semigroup:
$$S=\bigcup_{a\in A} N_a,$$
where $A$ is a finite set and $N_a=\langle a\rangle$ for $a\in A$.
We proved in \cite [Theorem 3.1]{abughazalah13} that the semigroup $S$ has the finite presentation 
 
 \begin{equation}
 \label{eq4a}
\big\langle A|\ a^k b=[\alpha(a,k,b,1)]^{\kappa(a,k,b,1)},\  (a,b\in A,\ k\in\{1,2,\dots,j\})\big\rangle,
 \end{equation}
for some $\alpha(a,k,b,1)\in A$ and $\kappa(a,k,b,1) \in\N$.

 We introduce the necessary lemmas from the paper \cite{abughazalah13} to add more information to the presentation \eqref{eq4a}.

\begin{lemma}[\cite{abughazalah13}, Lemma 2.4]
\label{lem1.2}
 If 
\[
a^px=b^r,\ a^{p+q}x=b^{r+s}
\]
for some $a,b\in A$, $x\in S$, $p,q,r\in\N$, $s\in\N_0$, then
\[
a^{p+qt}x=b^{r+st}
\]
for all $t\in\N_0$.
\end{lemma}

 \begin{lemma}
   \label{lem2}
  Let $a,c \in A, b\in S$. If $a^pb=c^{n_p}$ for two distinct values of $p$ then there exists an arithmetic progression $p+qn$, $r\in \mathbb{N}$, $s\in \mathbb{N}\cup \{0\}$ such that $a^{p+qn}b=c^{r+sn}$ for every $n\in\{0,1,2,.....\}$.
   \end{lemma} 
   \begin{proof}
   Since, $a^pb=c^{n_p}$ for two distinct values of $p$, then we have $q, r$, such that $a^qb=c^{n_q}$, $a^{r}b=c^{n_r}$ where $q\leq r$ and $r-q$ is as small as possible. Hence, by Lemma \ref{lem1.2}, $a^{r+n(r-q)}b=c^{{n_r}+n(n_r-n_q)}$ holds for every $n\in \mathbb{N}$.
   \end{proof}
   \begin{lemma}
   \label{lem3}
   Let $a,c\in A, b\in S$. Suppose $q\in \mathbb{N} $ is the smallest possible number such that $a^pb=c^r$, $a^{p+q}b=c^{r+s}$ for some $p,r\in \mathbb{N},s\in \mathbb{N}\cup\{0\}$ holds in $S$. Then if $i\in \{p+1,p+2,.....,p+q-1\}$ and $a^ib=d^t$ we have $d\neq c$.
   \end{lemma}   
     
   \begin{proof}
   Suppose to the contrary that $d=c$ then we have 
   \begin{enumerate}[$(i)$]
   \item $a^pb=c^r$;
   \item $a^ib=c^t$;
   \item $a^{p+q}b=c^{r+s}$.
   \end{enumerate}
   \textbf{Case 1.} If $t\geq r$ then from $(i),(ii)$ and Lemma \ref{lem1.2}, we obtain an arithmetic progression with a difference $i-p\leq q$, a contradiction.\\
   \textbf{Case 2.} If $t< r$ then from $(ii),(iii)$ and Lemma \ref{lem1.2}, we obtain an arithmetic progression with a difference $p+q-i\leq q$, a contradiction.
   \end{proof}  
   \begin{definition} 
\emph{   An arithmetic progression} is a sequence of the form $a^k, a^{k+q},a^{k+2q},\cdots$ where $a\in A,\ k,q \in \N$, (so that the difference of any two consecutive powers is constant), and we call it a \emph{minimal arithmetic progression} if $q$ is the smallest possible number such that $a^kb=c^r,\ a^{k+q}b=c^{r+s}$ for some $k,r\in \N,\ s\in \N\cup{0}$.
   \end{definition}
   \begin{definition}
   The \emph{ interval} on $N_a$ of \emph{length} $L$ is the set $$I=[a^t,a^{t+L}]=\{a^h\in N_a: t\leq h\leq t+L\}.$$
   \end{definition}
   \begin{lemma}
\label{lem1.3}
There exists $P\in\mathbb{N}$ such that the following holds. 
For every (not necessarily distinct) $a,b\in A$,  and every $x\in S$, if $a^r$ and $a^s$ ($r<s$) are the first two powers of $a$ such that $a^rx,a^sx\in N_b$ then $s-r\leq P$.
\end{lemma}

 \begin{proof}  
 Consider $x$ to be arbitrary but fixed. Within $N_a$ there are at most $n=|A|$ minimal arithmetic progressions by Lemmas \ref{lem1.2}, \ref{lem3}, one for each $N_b$, $b\in A$. So we have sets $A_{c_i}$ for $c_i \in \{c_1,c_2,\cdots,c_m\}\subseteq A$ where each $A_{c_i}$ is the set of elements $a^k$ such that $a^kx=c_i^r$, with the differences $d_1\leq d_2\leq \dots \leq d_s \leq \dots \leq d_m$ respectively. Thus, $N_a=H\cup A_{c_1}\cup A_{c_2}\cup \dots \cup A_{c_s}\cup \dots \cup  A_{c_m}$ where $H=\{a,a^2,\dots,a^p\}$ and $A_{c_s}=\{a^{p_s},a^{p_s+d_s},\dots\}$ such that $p_s=p+s$ for every $1\leq s\leq m$ and then $A_{c_1}\cup A_{c_2}\cup \dots \cup A_{c_s}\cup \dots \cup  A_{c_m}$ contains all but finitely many elements of $N_a$ which is $H$ by Lemma \ref{lem1.2}. Now, we prove that there exists $P \in \mathbb{N}$ not dependent on $x$, such that $d_s\leq P$ and this is sufficient since $a,b\in A$, $A$ is finite and by taking the maximum of $P$ over all $a,b$ will do for all. Let us consider an interval $I$ on $N_a$ of length $L=d_1d_2\dots d_{s-1}$ which occurs at the point $a^{p_M+1}$ where $p_M$ is the maximum power among $\{p_1,p_2,\dots,p_s,\dots,p_m\}$.\\
 
\textbf{Claim}. $I$ must contain at least one element from $A_{c_s}\cup A_{c_{s+1}}\cup \dots \cup A_{c_m}.$\\
 P{\footnotesize ROOF}.
Suppose to the contrary that all the elements in $I$ belong to $A_{c_1}\cup A_{c_2}\cup \dots \cup A_{c_{s-1}}$. Since $A_i$ is an arithmetic progression with a difference $d_i$ $(1\leq i\leq s-1)$, and since $d_i\vert L$ it follows that $A_{c_{i+L}}\subseteq A_{c_i}$. Hence, if $I\subseteq A_{c_1}\cup A_{c_2}\cup \dots \cup A_{c_{s-1}}$, it follows that $I \cdot  a^L\subseteq A_{c_1}\cup A_{c_2}\cup \dots \cup A_{c_{s-1}}$, and so $I \cdot a^{uL}\subseteq A_{c_1}\cup A_{c_2}\cup \dots \cup A_{c_{s-1}}$ for all $u\in \mathbb{N}$. Since $I$ is an interval of length $L$, it follows that $\bigcup\limits_{u\in \mathbb{N}}(I\cdot a^{uL})$ contains all but finitely many elements of $\mathbb{N}$. This contradicts the fact that $A_{c_s}$ is an infinite set disjoint from all $A_{c_1}, A_{c_2}, \dots , A_{c_{s-1}}$. Therefore the claim has been proved.\qed

\quad\quad Now we prove the lemma by induction on $s$. If $s=1$ then we choose $L=1$. Assume that the statement holds for every $k\leq s-1$. As a result of our claim, an interval $J$ of length $tL$ can be viewed as a disjoint union of $t$ intervals of length $L$. Each of the latter contains a elements from $ A_{c_s}\cup \dots \cup A_{c_m}$, and so $J$ contains at least $t$ such elements. Suppose that $d_s> L(m+1)$. So the interval $[a^r,a^{r+L(m+1)}]$ contains at least $m+1$ elements from $A_{c_s}\cup A_{c_{s+1}}\cup \dots \cup A_{c_m}$ and no elements from $A_{c_s}$. Then by using the pigeonhole principle, we conclude that two elements come from the same $A_{c_t}\ (s<t\leq m)$ with a difference less than $d_s$, a contradiction. Thus $d_s\leq L(m+1)\leq L(n+1)$. Since the number $L$ is dependent on $d_1,\dots,d_{s-1}$, none of them is dependent on $x$ by the induction hypothesis and by replacing $m$ by $n$ which is independent of $x$, we get $P=L(n+1)$ which does not depend on $x$.
\end{proof}

 The preceding result means that the differences of all minimal arithmetic progressions arising in Lemma \ref{lem1.2} are uniformly bounded.

In the next  lemma we prove that there is a uniform bound to how far arithmetic progressions can start.

\begin{lemma}
\label{lem1.4}
There exists $Q\in\mathbb{N}$ such that the following holds.
For every $a,b\in A$,  and every $x\in S$, if $a^r$ and $a^s$ ($r<s$) are the first two powers of $a$ such that $a^rx,a^sx\in N_b$ then $r\leq Q$.
\end{lemma}

 \begin{proof}  
 Assume the opposite, i.e. that the start of an arithmetic progression can occur arbitrarily far into $N_a$, say beyond $T\geq (n+1)P$, where $P$ is the constant in Lemma \ref{lem1.3} and $n=|A|$. That means
  \begin{eqnarray}  
  \label{eq6}
  a^Tx=b^p,\ \ a^{T+d}x=b^q
  \end{eqnarray}
  Since the difference $d\leq P$ by Lemma \ref{lem1.3}, the $n$ numbers $T-dk\ (1\leq k\leq n)$ are all positive. By using the pigeonhole principle, there are two distinct powers $a^{T-hd},a^{T-kd}$ where (without loss of generality) we assume $h>k$, such that $a^{T-hd}x,a^{T-kd}x$ belong to the same $N_c$ with a difference $(h-k)d$ where $b\neq c$ and that is clear because $a^{T-hd},a^{T-kd}$ appear before $a^{T},a^{T+d}$ where they are the first two powers such that \eqref{eq6} holds. Therefore, $a^{T+2(h-k)d}x \in N_c\ $ but this element also belongs to $N_b$ where the power $T+2(h-k)d\in \{T,T+d,\dots,T+md,\dots\} (m\in \mathbb{N}) $, a contradiction.
  \end{proof}
  
\begin{lemma}
\label{lem7}
As $x=b^s$ ranges over all of $S$, only finitely many arithmetic progressions arise in Lemma \ref{lem1.2}.
\end{lemma}

\begin{proof}  
Immediate consequence of Lemmas \ref{lem1.3}, \ref{lem1.4} in which all these arithmetic progressions start within a bounded range and their periods are bounded as well.
\end{proof}

\section{Decidability for $S$}

\subsection{Word problem}
A semigroup $S$ generated by a finite set $A$ has  soluble word problem (with respect to $A$) if there exists an algorithm which, for any two words $u,v\in  A^+$, decides whether the relation $u = v$ holds in $S$ or not. For finitely generated semigroups it is easy to see that solubility of the word problem does not depend on the choice of (finite) generating set for $S$. We write $w_1 \equiv w_2$ if the words $w_1$ and $w_2$ are identical, and $w_1=w_2$ if they represent the same element of $S$.
\\
\begin{remark}
It is well known that finitely presented residually finite semigroups have soluble word problem \cite{golubov75}, and from our results in \cite{abughazalah13}, the semigroup under consideration in this paper is finitely presented and residually finite and then has soluble word problem. However, we give a concrete algorithm which is much more efficient than the generic one in the following theorem.
\end{remark}
\begin{thm}
\label{thm 9.2.1}
Every semigroup which is a disjoint union of finitely many copies of the free monogenic semigroup has soluble word problem.
\end{thm}
\begin{proof}
 Let $S=\bigcup_{a\in A} N_a$, and $N_a=\langle a\rangle$. Thus the \textbf{Algorithm} is as follows:\\
\textbf{Input}: $u,v\in A^+$ and $u\equiv x_1^{i_1}x_2^{i_2}\cdots x_m^{i_m}$ and $v \equiv y_1^{j_1}y_2^{j_2}\cdots y_n^{j_n}$, where $x_k,y_l \in A$ for every $1\leq k \leq m$ and $1\leq l\leq n$.\\\\
\textbf{Output}: $u=v$ or $u\neq v$ in $S$.
\\

 The idea of the \textbf{Algorithm} is to reduce the word $u$ to an equivalent word $a^k$ for some $a \in A$ and $ k \in \N$.
\begin{enumerate}[\textbf{Step} 1.]
\item 
\begin{lemma}
If we have a  presentation on $S$ of the form
 \begin{equation}
 \label{eq5a}
\big\langle A|\ a^k b=[\alpha(a,k,b,1)]^{\kappa(a,k,b,1)},\  (a,b\in A,\ k\in\{1,2,\dots,j\})\big\rangle,
 \end{equation}
for some $\alpha(a,k,b,1)\in A$ and $\kappa(a,k,b,1) \in\N$,
then there is an algorithm that reduces any word of the form $a^ib$ to a word of the form $c^j$ where $a,b,c \in A$ and $j \in \N$.
\end{lemma}
\begin{proof} 
We specify the presentation \eqref{eq5a} as follows. Firstly, notice that the relations in the presentation are of the form $x^iy=z^j$ where $x,y,z\in A$ and $i,j\in \N$ and thus we have at most $n(n-1)$ minimal arithmetic progressions in which we get at most $n(n-1)$ differences. Take the the least common multiple ($LCM$) of all these differences $D$. Thus $$R_{a,b}'=\bigcup\big\{a^ib=[{\alpha(a,i,b,1)}]^{\kappa(a,i,b,1)}\ :i=1,\dots,r(a,b)\big\},$$
Where $r(a,b)\leq Q=(n+1)P$ from \ref{lem1.3}, \ref{lem1.4}.
\begin{equation}
\label{eq2}
R_{a,b}=\bigcup\limits_{k=r(a,b)+1}^{r(a,b)+D}\big\{{a^kb=[\alpha(a,k,b,1)}]^{\kappa(a,k,b,1)},\ a^{k+D}b=[{\alpha(a,k+D,b,1)}]^{\kappa(a,k+D,b,1)}\big\},
\end{equation}
and then we get the required presentation as
$$R=\bigcup\limits_{a,b\in A}(R_{a,b}'\cup R_{a,b}),$$
where $k=lD$, $l$ is any natural number. Notice that from \eqref{eq2} we have 
\begin{equation}
\label{eq3}
a^{k+D}b=a^Da^kb=a^D[\alpha(a,k,b,1)]^{\kappa(a,k,b,1)}=[\alpha(a,k+D,b,1)]^{\kappa(a,k+D,b,1)}
\end{equation}

So within $N_a$ we have $P_t$ arithmetic progressions, where $t$ is the remainder of division of $r(a,b)+q$ by $D$ for every $q\in\{1,2,\cdots,D\}$  as follows:
$$P_0=\{a^{r(a,b)+1},a^{r(a,b)+1+D},a^{r(a,b)+1+2D},\dots\},$$
$$P_1=\{a^{r(a,b)+2},a^{r(a,b)+2+D},a^{r(a,b)+2+2D},\dots\},$$
$$\vdots$$
$$P_{D-1}=\{a^{r(a,b)+D},a^{r(a,b)+2D},a^{r(a,b)+3D},\dots\}.$$
Thus for every $i \in \N$ every word of the form $a^ib$ is reduced to a word of the form $c^j$.
\end{proof}
\item 
\begin{lemma}
\label{lem 9.2.1}
For $a,b \in A$ and $s \in \N$, a finite number of applications of relations in $R$ transforms $a^sb$ to $[\alpha(a,r(a,b)+t+1+fD,b,1)]^{\kappa(a,r(a,b)+t+1+fD,b,1)}$ where $\alpha(a,r(a,b)+t+1+fD,b,1) \in A$ and $\kappa(a,r(a,b)+t+1+fD,b,1) \in \N.$
\end{lemma}
\begin{proof}
 If the relation $$a^sb=[\alpha(a,r(a,b)+t+1+fD,b,1)]^{\kappa(a,r(a,b)+t+1+fD,b,1)}$$ belongs to $R$, we are done. Now, suppose that the given relation does not appear in $R$, that means $s>k$ where $k=lD$ for some $l$ and then $s=hD+t$ where $0\leq t<D$ and thus $a^s \in P_t$. Notice that $P_t$ starts with the two elements $a^{r(a,b)+(t+1)},\ a^{r(a,b)+(t+1+D)}$ and by doing some calculations as follows:\\
First we know that $$s=hD+t,$$ and $$s-r(a,b)-t-1=hD+t-r(a,b)-t-1=fD$$ for some $f$. Thus, 
$$hD=fD+r(a,b)+1.$$
So,$$s=r(a,b)+t+1+fD,$$ which means that $a^s$ is in the $f$ position. Hence, 
\begin{align}
a^sb&=a^{fD+r(a,b)+t+1}b\notag\\&\equiv\underbrace{a^Da^D\cdots a^D}_{f\ \text{times}} a^{r(a,b)+t+1}b\notag\\&=\underbrace{a^Da^D\cdots a^D}_{f \ \text{times}} [\alpha(a,r(a,b)+t+1,b,1)]^{\kappa(a,r(a,b)+t+1,b,1)}&(\text{by}\  \eqref{eq2})\notag\\&=\underbrace{a^D\cdots a^D}_{(f-1)\ \text{times}} [\alpha(a,r(a,b)+t+1+D,b,1)]^{\kappa(a,r(a,b)+t+1+D,b,1)}& (\text{by}\  \eqref{eq3})\notag\\& \vdots \notag \\&= [\alpha(a,r(a,b)+t+1+fD,b,1)]^{\kappa(a,r(a,b)+t+1+fD,b,1)}.\nonumber
\end{align}
Therefore, we can obtain $[\alpha(a,r(a,b)+t+1+fD,b,1)]^{\kappa(a,r(a,b)+t+1+fD,b,1)}$ in finitely many steps.
\end{proof}

\item Transform $u$ to its normal form as follows:
\begin{align}
u&\equiv x_1^{i_1}x_2^{i_2}\cdots x_m^{i_m}\notag\\&\equiv (x_1^{i_1}x_2)x_2^{{i_2}-1}\cdots x_m^{i_m}\notag\\&=x_{i_{12}}^{i_{12}}x_2^{{i_2}-1}\cdots x_m^{i_m} &&\text{(by Lemma \ref{lem 9.2.1})}\notag\\&\equiv (x_{i_{12}}^{i_{12}}x_2)x_2^{{i_2}-2}\cdots x_m^{i_m}.\nonumber
\end{align}
So, by taking the first power $x_1^{i_1}$ with the next element $x_2$ and doing this $i_2$ steps, we get rid of $x_2^{i_2}$ and using the same process with all $x_3^{i_3},x_4^{i_4},\cdots,x_m^{i_m}$, we end up with $x_I^{I_M}$ after $i_2+i_3+\dots+i_m$ steps. So we have $u=x_I^{I_M}.$
\item Transform $v$ to its normal form $x_J^{J_N}$ analogously to step 3 .
\item If $I=J$ and $I_M=J_N$ then $u=v$, otherwise $u\neq v$. 
\end{enumerate}
Therefore, $S$ has soluble word problem.
\end{proof} 
\subsection{Subsemigroup membership problem}
Let $S$ be a finitely generated semigroup. We say that $S$ has a soluble subsemigroup membership problem if there is an algorithm that takes as input a finite set  $Y=\{y_1,y_2,\cdots,y_k\}\subseteq S$ and an element $x \in S$ and decides whether $x$ is in the subsemigroup $T$ generated by $Y$.

\quad\quad Now we introduce necessary well-known theorems about subsemigroups of the natural number semigroup $\mathbb{N}$. We will use these theorems to devise an algorithm to solve the subsemigroups membership problem for the semigroup under consideration.
\begin{thm}[\cite{higgins1972}, Theorem 1]
\label{thm 9.1.1}
Let $S$ be a subsemigroup of $\mathbb{N}$, then
\begin{enumerate}[$i)$]
\item There is $s\in\mathbb{N}$ such that for $n\geq s,\ n\in S$, or
\item There is $n\in \mathbb{N},\ n>1$ such that $n$ is a factor of all $s\in S$.
\end{enumerate}
\end{thm}
We prove this theorem as the proof itself  leads us to Corollary \ref{cor 9.3.1}.  \\
P{\footnotesize ROOF}. Assume that there exist $s_1,s_2,\dots,s_m\in S$ such that the g.c.d of the collection $(s_1,s_2,\dots,s_m) $ is 1. Let $S'$ be the subsemigroup of $\mathbb{N}$ generated by $\{s_1,s_2,\dots,s_m\}$, notice that $S'\subseteq S$. Let $s=2s_1s_2\dots s_m$ and for $b>s$, since the g.c.d of $(s_1,s_2,\dots,s_m)$ is 1, we may find integers $\alpha_1,\alpha_2,\dots,\alpha_m$ such that $\alpha_1s_1+\dots+\alpha_ms_m=b$.\\ Hence there exist integers $q_i$ and $r_i$ such that $\alpha_i=q_is_1\dots s_{i-1}s_{i+1}\dots s_m+r_i$ where $0<r_i\leq s_1\dots s_{i-1}s_{i+1}\dots s_m\ (i=2,3,\dots,m)$. Now put $$\beta_1=\alpha_1+(q_2+\dots+q_m)s_2s_3\dots s_m,\ \beta_i=r_i,\ (i=2,3,\dots,m).$$ Thus $b=\beta_1s_1+\beta_2s_2+\dots + \beta_ms_m.$ Note that $\beta_i>0$ for $i=2,3\dots,m$. But since $$\beta_2s_2+\dots+\beta_m s_m=r_2s_2+\dots+r_ms_m\leq 2s_1s_2\dots s_m<b,$$clearly $\beta_1>0$.\qed 

\quad\quad Thus there are two types of subsemigroups of $\mathbb{N}$. The first type contains all natural numbers greater than some fixed natural number,
and will be called relatively prime subsemigroups of $\mathbb{N}$. The second type is a fixed integral multiple of a relatively prime subsemigroup.

\begin{cor}
\label{cor 9.1.1}
Every subsemigroup of $\mathbb{N}$ is finitely generated.
\end{cor}
\begin{remark}
 This corollary is well known and here is an easy proof.
 \end{remark}
 \begin{proof}
  Suppose that $S$ is a subsemigroup of  $\mathbb{N}$ and the greatest common divisor of $S$ is 1. Thus the generating set for $S$ is $S \cap \{1,2,\dots, 2k\}$ where $k\in \mathbb{N}$ such that for every $n\geq k \ :\ n\in S$. Indeed this is so because if $m>2k $ then $m=qk+f$. Thus $m=(q-1)k+k+f$ where $k+f\in S\cap \{1,2,\dots,2k\}$.
 \end{proof}

\textbf{Fact}: If $S$ is a subsemigroup of $\mathbb{N}$ then the greatest common divisor g.c.d of $S$ is the g.c.d of the generator set of $S$.\\

\begin{cor}
\label{cor 3.3*}
Every subsemigroup of $\mathbb{N}$ has the form $$F\cup D_{\mathcal{N},d},$$ where F is a finite set and $ D_{\mathcal{N},d}=\{da\ :\ a\geq\mathcal{N}\}$.
\end{cor}
\begin{definition}
Suppose that the semigroup $S$ is generated by $\{n_1,n_2,\cdots,n_k\}$. If there exist two elements $d,\mathcal{N}\in S$ and a set $F\subseteq S$ such that 
$$F=S\cap\{1,2,\cdots,{\mathcal{N}}-1\};$$
$$S\cap \{{\mathcal{N}},{\mathcal{N}}+1,\cdots\}=\{dk\ :\ k\in\mathbb{N},\ dk\geq \mathcal{N}\},$$
then we say that $S$ is defined by the \emph{triple $[d,\mathcal{N},F]$}.
\end{definition}
\begin{cor}
\label{cor 9.3.1}
Suppose that $S$ is a subsemigroup of the natural number semigroup $\mathbb{N}$. Suppose that $S$ is generated by $n_1,n_2,\cdots,n_k$. Then $S$ is defined by the triple $[d,\mathcal{N},F]$ where $d$ is the greatest common divisor of $\{n_1,n_2,\cdots,n_k\}$, $$\mathcal{N}=2dn_1n_2\cdots n_k,$$ and $$F\subseteq\{1,2,\cdots,\mathcal{N}-1\}.$$ \end{cor}
\begin{proof}
 Follows immediately from Theorem \ref{thm 9.1.1} and Corollary \ref{cor 9.1.1}.
 \end{proof}
\begin{cor}
\label{cor 9.3.2}
Suppose that $S$ is a subsemigroup of the free monogenic semigroup $N$. Suppose that $S$ is generated by $a^{n_1},a^{n_2},\cdots,a^{n_k}$. Then $S$ is defined by the triple $[d,\mathcal{N},F]$ where $d$ is the greatest common divisor of $\{a^{n_1},a^{n_2},\cdots,a^{n_k}\}$, $$\mathcal{N}=a^2da^{n_1}a^{n_2}\cdots a^{n_k},$$ and $$F\subseteq\{a,a^2,\cdots,a^{\mathcal{N}-1}\}.$$ \end{cor}
\begin{proof}
 Directly by Corollary \ref{cor 9.3.1}.
\end{proof}

\quad\quad After understanding how subsemigroups of $\mathbb{N}$ behave we are ready to start designing the algorithm. Since  $$S=N_1\cup N_2\cup\cdots\cup N_n,$$ and $T$ is a subsemigroup of $S$, then $$T=T_1\cup T_2\cup\dots\cup T_m,$$ where $T_i\leq N_i$ for every $i\in \{1,2,\cdots,m\},\ m\leq n$. Consequently, the generator set for $T$ is $$A_T=\bigcup\limits_{i\in\{1,2,\cdots,m\}} A_{T_i},$$ where $A_{T_i}$ is the generator set of $T_i$ for every $i\in\{1,2,\cdots,m\}$. Thus $T$ is finitely generated (\cite{araujo01}, Proposition 3.1).
 \begin{lemma}
\label{lem 9.3.1}
 Suppose that the subsemigroup $U_j=\langle N_j\cap A_T\rangle$ is defined by the triple $\big[d_j,\mathcal{N}_j ,F_j\big].$ Then there is an algorithm which takes arbitrary $U_i,U_j$ and $b\in A_T$ and tests whether $$U_ib\cap N_j\subseteq U_j$$ or not. 
\end{lemma}  
\begin{proof}
 Let $a_j^r\in U_ib\cap N_j$. Then $$a_j^r\in U_j\iff a_j^r\in F_j\ \text{or}\ a_j^r=a^{d_jh_j}_j\ \text{for some}\ d_jh_j\geq d_jt_j\ \text{where} \ d_jt_j=\mathcal{N}_j,$$
by Corollary \ref{cor 9.3.1}.
\end{proof}

\begin{thm}
\label{thm 9.3.1}
Every semigroup which is a disjoint union of finitely many copies of the free monogenic semigroup has a soluble subsemigroup membership problem.
\end{thm}
\begin{proof}
 Let $S$ be such a semigroup. Then the \textbf{Algorithm} is as follows:\\
\textbf{Input}. A finite set $A_T \subseteq S$, specified as normal form words over the generating set, and an element $x\in S$, specified as a normal form word.
\\\\
\textbf{Output}. Whether $x\in T$ where $T$ is the subsemigroup of $S$ generated by $A_T$.
\begin{enumerate}[\textbf{Step} 1.]
 \item 
 Take $U_i=\langle A_i \rangle =\langle N_i \cap A_t\rangle$, which means that $U_i$ is a finitely generated subsemigroup of $N_i$ and then is defined by the triple $[d_i, \mathcal{N}_i,F_i]$ by Corollary \ref{cor 9.3.2}.
  Now check if $$(U_1,U_2,\dots,U_m)= T$$
 where $$(U_1,U_2,\dots,U_m)=U_1\cup U_2\cup \cdots \cup U_m,$$ which means check whether $$U_ix\subseteq \bigcup\limits _{i=1}^{m} U_i\  \text{for every}\ i\in \{1,2,\cdots,m\}\  \text{ and for every}\  x\in A_T,$$ by Lemma \ref{lem 9.3.1}. If yes then go to step 4. If there was $a_i^{r_i}x=a_j^{r_j}$ and $a_j^{r_j}\not\in U_j$ then go to step 2.
 \item Add the missing element $a_j^{r_j}$ to $U_j$ and then we have $$U_j^{(+1)}=\langle A_{U_j}\cup a_j^{r_j}\rangle,\ \text{ which is defind by the triple}\ [d_j^{(+1)}, \mathcal{N}_j^{(+1)}, F_j^{(+1)}].$$  Notice that by adding $a_j^{r_j}$ to $U_j$ we reduce the gaps in $F_j$ or we reduce the difference $d_j$ by Corollary \ref{cor 3.3*}. Thus we get the new description
\begin{equation}
\label{eq 9.1}
\big(U_1,U_2,\cdots,U_{j-1},U_j^{(+1)},U_{j+1},\cdots,U_m\big),
\end{equation}
 \item We start again with the new description \eqref{eq 9.1} and we keep adding these missing elements with all $i\in \{1,2,\cdots,m\}$.
 And then we reach to the final description
$$\big(U_1^{(+s_1)},U_2^{(+s_2)},\cdots,U_j^{(+s_j)},\cdots,U_m^{(+s_m)}\big)=T.$$
Which means that $U_j^{(+s_j)}b\subseteq \bigcup\limits_{i=1}^{m} U_i^{(+s_i)}$ for every $b\in A_T$ and for every $j\in\{1,2,\cdots,m\}$ and that because as we explained before each $U_j$ is defined by the triple $\big[d_j,\mathcal{N}_j,F_j\big]$. So if we add an element $a_j^{r_j}$ to $U_j$ that means, by Corollary \ref{cor 3.3*}, we reduce the gaps in $F_j$ and they are finite, or we reduce the difference $d_j$ and we can do this just finitely often. Thus we add finitely many elements in each $U_j$, which implies that this process terminates. So now each $U_j^{(+s_j)}$ is defined by the triple $$\big [d_j^{(+s_j)},\mathcal{N}_j^ {(+s_j)},F_j^{(+s_j)}\big].$$
\item  If we were given $x=a_h^{r_h}\in S$ and we want to see if $x\in T$ or not then we just take this element and see in $U_h^{(+s_h)}$ if $$a_h^{r_h}\in F_h^{(+s_h)},$$ or$$r_h=d_h^{(+s_h)}k\ \text{for some}\ d_h^{(+s_h)}k\geq d_h^{(+s_h)}t\ \text{where} \ d_h^{(+s_h)}t=\mathcal{N}_h^{(+s_h)},$$
  then $x\in T$ otherwise $x\not \in T$.
\end{enumerate}
\end{proof}
\section*{Acknowledgment}
I would like to thank the reviewer for their detailed comments and constructive suggestions for the manuscript. The author wants to thank Prof Nik Ru\v skuc (University of St Andrews) for his many helpful contributions to this work. And to thank Prof Pavel Etingof (MIT) for his useful comments and suggestions on this paper.

\end{document}